\documentclass{amsart}

\usepackage{amssymb}
\usepackage{mathrsfs}
\usepackage{booktabs}
\usepackage{hyperref}
\newtheorem{theorem}{Theorem}[section]
\newtheorem{thm}[theorem]{Theorem}
\newtheorem{lem}[theorem]{Lemma}

\newtheorem{cor}[theorem]{Corollary}

\newtheorem{que}[theorem]{Question}

\theoremstyle{definition}

\newtheorem{example}[theorem]{Example}

\theoremstyle{remark}
\newtheorem{rem}[theorem]{Remark}

\numberwithin{equation}{section}

\newcommand{\FF}[1]{\mathbb F_{#1}}
\newcommand{\ZZ}{\mathbb Z}

\newcommand{\nor}{\vartriangleleft}

\newcommand{\Hom}{\operatorname{Hom}}
\newcommand{\End}{\operatorname{End}}
\newcommand{\aut}{\operatorname{Aut}}
\newcommand{\Aut}{\operatorname{Aut}}
\newcommand{\codim}{\operatorname{codim}}

\newcommand{\id}{\operatorname{id}}
\newcommand{\A}{\mathcal A}
\newcommand{\B}{\mathcal B}

\begin{document}
\title[On $p$-groups with an abelian normal subgroup of rank $k$]{On $p$-groups with a maximal elementary 
  abelian normal subgroup of rank $k$} 

\author{Zolt\'an  Halasi} 
\address{ Department of Algebra and Number Theory,
  E\"otv\"os University, P\'azm\'any P\'eter s\'et\'any 1/c, H-1117,
  Budapest, Hungary \and Alfr\'ed R\'enyi Institute of Mathematics,
  Re\'altanoda utca 13-15, H-1053, Budapest, Hungary\newline
  ORCID: \url{https://orcid.org/0000-0002-1305-5380}
} 
\email{zhalasi@caesar.elte.hu and halasi.zoltan@renyi.hu}

\author{K\'aroly Podoski}
\address{Alfr\'ed R\'enyi Institute of Mathematics,
  Re\'altanoda utca 13-15, H-1053, Budapest, Hungary}
\email{podoski.karoly@renyi.hu}

\author{L\'aszl\'o Pyber}
\address{Alfr\'ed R\'enyi Institute of Mathematics,
  Re\'altanoda utca 13-15, H-1053, Budapest, Hungary}
\email{pyber.laszlo@renyi.hu}

\author{Endre Szab\'o}
\address{Alfr\'ed R\'enyi Institute of Mathematics,
  Re\'altanoda utca 13-15, H-1053, Budapest, Hungary}
\email{szabo.endre@renyi.hu}

\date{\today}

%\keywords{}
%\subjclass[2010]{}
\thanks{This work on the project leading to this application has
  received funding from the European Research Council (ERC) under the
  European Union's Horizon 2020 research and innovation programme
  (grant agreement No. 741420).
  The first and the third authors were partly supported by the
  National Research, Development and Innovation Office (NKFIH) Grant
  No.~K138596.}
%======================================================================
\begin{abstract}
  There are several results in the literature concerning $p$-groups
  $G$ with a maximal elementary abelian normal subgroup of rank $k$
  due to Thompson, Mann and others. 
  Following an idea of Sambale we obtain bounds
  for the number of generators etc. of a $2$-group $G$ in terms of $k$,
  which were previously known only for $p>2$. 
  We also prove a theorem that is new even for odd
  primes. Namely, we show that if $G$ has a maximal elementary abelian
  normal subgroup of rank $k$, then for any abelian
  subgroup $A$ the Frattini subgroup $\Phi(A)$ can be generated by
  $2k$ elements ($3k$ when $p=2$). The proof of this rests upon the
  following result of independent interest: If $V$ is an
  $n$-dimensional vector space, then any commutative subalgebra of
  $\End(V)$ contains a zero algebra of codimension at most $n$.
\end{abstract}
\maketitle
%---------------------------------------------------------
\section{Introduction}
For a finite $p$-group $G$ we denote by $d(G)$ the size of (any)
minimal set of generators for $G$. Then the $p$-rank of $G$ (denoted
by $r(G)$), the normal $p$-rank of $G$ (denoted by $nr(G)$) and the
sectional $p$-rank of $G$ (denoted by $sr(G)$) are defined as
\begin{align*}
  r(G)  &=\max\{d(A)\,|\,A\leq G,\ A\textrm{ is abelian}\},\\
  nr(G)&=\max\{d(A)\,|\,A\nor G,\ A\textrm{ is an abelian}\},\\
  sr(G) &=\max\{d(H/K)\,|\,K\nor H\leq G,\ H/K\textrm{ is abelian}\}.
\end{align*}
Note that $sr(G)$ equals the maximum of the generator numbers of all
the subgroups of $G$.

These parameters were much investigated in the past.  The results of
Blackburn and MacWilliams (see \cite{blackburn1}, \cite{blackburn2},
\cite{macwilliams-lowrank}) concerning $p$-groups of very low rank
played an important role in the proof of the Classification Theorem of
Finite Simple Groups (see also Janko \cite{janko}).
A natural question that arises is that knowing
$r(G)$ or $nr(G)$ what can be said about $sr(G)$.  By a classical
result of Thompson if $p$ is odd and $nr(G)$ is at most $k$ then any
subgroup of $G$ can be generated by at most $\frac{k(k+1)}{2}$
elements.
%In particular,
%what is the connection between $r(G)$ and $nr(G)$? Of course,
%$nr(G)\leq r(G)$ always holds, but $nr(G)$ can be strictly smaller
%than $r(G)$ (see Question \ref{que:d(A)_leq_2max_d(E)} and Remark
%\ref{rem:d(A)_leq_2max_d(E)}).
%On the other hand, $r(G)$ is bounded by a quadratic
%function of $nr(G)$. 
Thompson's result has been later improved by MacWilliams and an
analogous bound has been obtained by Mann for $p=2$.

In fact, the following, much stronger results were proved.
%-----------------
\begin{thm}\label{thm:Th-Mann}
  Let $G$ be a finite $p$-group and let $E$ be a maximal elementary
  abelian normal subgroup of $G$. If $d(E)=k$, then
  \begin{enumerate}
  \item $sr(G)\leq \frac{k(k+1)}{2}$ for $p$ odd (Thompson \cite[III,
    12.3 Satz]{huppert});
  \item $sr(G)\leq \frac{k(k+4)}{4}$ for $p$ odd
    (MacWilliams \cite[Theorem B]{macwilliams});
  \item $sr(G)\leq k^2+\frac{k(k+1)}{2}$ for $p=2$ (Mann \cite[Theorem
    B]{mann}).
  \end{enumerate}
\end{thm}
% -----------------
%\begin{rem}
%  By a result of Macwilliams \cite[Theorem B]{macwilliams},
%  part (1) can be sharpened to $d(H)\leq \frac{k(k+4)}{4}$.
%\end{rem}
%-----------------
\begin{rem}\label{rem:Th-Mann-Macwilliams}\leavevmode
  The above cited theorems are formed with the stronger assumption
  ``each normal abelian subgroups of $G$ can be generated by $k$
  elements'' (i.e. that $nr(G)\leq k$), but it can be easily checked
  that the proofs of \cite[III, 12.3 Satz]{huppert}), \cite[Theorem
  B]{macwilliams} and of \cite[Theorem B]{mann} only use the existence
  of an $A\leq G$ which is maximal among the normal abelian subgroups
  of exponent $p$ or $4$ such that $d(A)\leq k$.
\end{rem}
%----------------- 
In particular, under the hypothesis of this theorem, every abelian
subgroup $A\leq G$ is the product of at most $O(k^2)$ many cyclic
groups. Note that apart from the implied constant multiple this is the
best possible bound for $d(A)$ as the following example shows.
%------------------
\begin{example}\label{ex:Th-Mann_best}
  Let $V$ be a $k=2m$ dimensional vector space over $\FF p$ and let
  $V_1\leq V$ be an $m$-dimensional subspace of $V$. Let us define
  \[
    H=\{\varphi\in GL(V)\,|\, \varphi_{V_1}=\id_{V_1},\ \varphi_{V/V_1}=
    \id_{V/V_1}\}\textrm{\quad and\quad} G=V\rtimes H,
  \]
  with the natural action of $H$ on $V$. Then
  both $V$ and $V_1\times H$ are maximal normal abelian subgroups
  of $G$ with $d(V)=k$ and $d(V_1\times H)=\frac{k^2}{4}+\frac{k}2$. 
\end{example}
% -------------------
The significant part of MacWilliams' improvement was to show that
if $p$ is an odd prime and $P$ is any $p$-subgroup of $GL(n,p)$, 
then $d(P)\leq\frac{n^2}{4}$ holds, see \cite[Theorem A]{macwilliams}.
(Note that Thompson's argument only uses the trivial bound
$d(P)\leq \frac{n(n-1)}{2}$.)

By a modification of her proof, the same result can be achieved
for $p=2$, as well.
%--------------------
\begin{thm}\label{thm:Macwilliams_p=2}
  Let $G$ be any $2$-subgroup of $GL(n,2)$.
  Then $G$ can be generated by at most $\frac{n^2}{4}$ elements.
\end{thm}
%--------------------
%This already gives an improvement of Mann's bound for $d(G)$.  
This result allow us to give a nearly optimal bound for the maximum of
the generating numbers of all subgroups of $GL(n,p)$.  On the one
hand, results of Lucchini \cite[Theorem 1]{lucchini} and Guralnick
\cite[Theorem A]{guralnick} say that if $G$ is a finite group and for
every prime $r\mid |G|$ the Sylow $r$-subgroups of $G$ can be
generated by $d$ elements then $G$ can be generated by $d+1$ elements.
On the other hand, it was proved by Isaacs \cite[Theorem A]{isaacs} that if
$r\geq 3$ is a prime different from the characteristic of a field $K$, then any
finite $r$-subgroup of $GL(n,K)$ can be generated by $n$ elements.
Combining these results with
Theorem \ref{thm:Macwilliams_p=2}, we obtain the following.
%--------------------
\begin{thm}\label{thm:sr(GL(n,p))}
  Every subgroup of $GL(n,2)$ can be generated by at most 
  $\frac{n^2}4+1$ elements. 
\end{thm}
%--------------------
For $p$ odd, the same result for $GL(n,p)$ 
already appears in \cite[p.~199]{pyber}.

In \cite[Remark 2.7]{babai-goodman}, Babai and Goodman claim that
if $|G|=p^n$ and $H$ is any $p$-subgroup of $\aut(G)$,
then $d(H)\leq \frac{1}{3}n^2$ follows from the result of MacWilliams for
$p>2$ and remark  that
they do not know whether such an estimate also holds for $p=2$.
As another consequence of Theorem \ref{thm:Macwilliams_p=2} we show
that it does indeed hold. In fact, almost the same estimate can be verified
for any subgroup of $\Aut(G)$.
% --------------------------------------
\begin{cor}\label{cor:babai-goodman}
  If $G$ is any $p$-group of order $p^n$, then
  $sr(\Aut(G))\leq \frac{1}{3}n^2+1$. Furthermore, every $p$-subgroup
  of $\Aut(G)$ can be generated by at most $\frac{1}3 n^2$
  elements.
\end{cor}
%----------
\begin{proof}
  Let $H$ be any subgroup of $\Aut(G)$. The action of $H$ on $G$
  induces and action of $H$ on the $\FF p$-vector space $G/\Phi(G)$.
  Let $\tau:H\to \Aut(G/\Phi(G))$ be the associated homomorphism and
  $K=\ker(\tau)\nor H$.  Let $k$ be the dimension of
  $G/\Phi(G)$. Then $\tau(H)$ embeds into $GL(k,p)$ so
  $d(H/K)=d(\tau(H))\leq \frac{1}4 k^2+1$ by \cite[p.~199]{pyber} 
  and by Theorem \ref{thm:sr(GL(n,p))}.

  On the other hand, by a result of Hall (\cite[Section 1.3,
  p.~37-38.]{hall}, \cite[Chapter 2, Theorem 1.17]{suzuki}), $K$ is a
  $p$-group of order at most $p^{(n-k)k}$, so we have
  $d(K)\leq k(n-k)$. Thus,
  \begin{equation}\tag{Eq.~1}\label{eq:sr(GL(n,p))}
    d(H)\leq d(H/K)+d(K)\leq {k^2/4+1}+{k(n-k)}=nk-\frac{3}4 k^2+1\leq
    \frac{1}3n^2+1.
  \end{equation}
  Therefore, $sr(\Aut(G))\leq \frac{1}3 n^2+1$, as claimed.

  If $H$ is a $p$-subgroup of $\Aut(G)$, then we can use \cite[Theorem
  A]{macwilliams} and Theorem \ref{thm:Macwilliams_p=2} to bound
  $d(H/K)$ in \ref{eq:sr(GL(n,p))} by $k^2/4$.
\end{proof}
% ---------------------------------------
Using Theorem \ref{thm:Macwilliams_p=2} for $p=2$ 
and other results of this paper, we improve Theorem 
\ref{thm:Th-Mann}(3) as follows.
% ------------------
\begin{thm}\label{thm:Th-Mann_p=2}
  Let $G$ be a finite $2$-group and let $E$ be a maximal
  elementary abelian normal subgroup of $G$. If $d(E)=k$, and $H$ is
  any subgroup of $G$, then $d(H)\leq 2k+\frac{1}4 k^2$. 
\end{thm}
% -----------------
Note that by Example \ref{ex:Th-Mann_best} this bound is 
almost optimal.

By an old result of Mann and Su \cite{mann-su}, if $M$ is a compact
manifold, then any elementary abelian $p$-group acting faithfully on
$M$ by homeomorphisms has rank at most $f(M)$, where $f(M)$ depends
only on $M$ (and does not depend on the prime $p$). In the work of the
third and fourth authors with Csik\'os \cite{csikos-pyber-szabo} the
following consequence of the above results is used.
%-----------
\begin{cor}
  If every elementary abelian subgroup of a finite group $G$ has rank at
  most $k$, then each subgroup $H$ of $G$ can be generated by at most
  $\frac{1}{4}k^2+2k+1$ elements.
\end{cor}
%-----------
This result is the starting point for obtaining a structural
description of finite groups acting on compact manifolds.

% --------------------
Note that Ol'shanskii \cite{olshanskii} has given a probabilistic
construction of $p$-groups $G$ of nilpotency class $2$ with $r(G)=k$
and $d(G)\geq (k^2 -9)/8$.  In section 3 we will use his method to
show the following:

\begin{thm}\label{thm:Endre_konstrukcio}
  For any prime number $p$ and positive integers $r,n, k$ with
  $k(k-1)>2n$ there is a $p$-group $G$ and $G'\leq N\leq Z(G)$ such that
  $G/N\simeq C_{p^r}^n$ and $G$ does not contain a subgroup 
  isomorphic to $C_{p^r}^{2k}$.
\end{thm}

%-------------------------------------------
For the remainder,
for any natural number $t$, we use the notation 
\[
\Omega_t(G)=\langle x\in G\,|\,x^{p^t}=1\rangle
\textrm{\ and\ }
\mho_t(G)=\langle x^{p^t}\,|\,x\in G\rangle.\]
Note that if $G$ is abelian, then
\[
\Omega_t(G)=\{x\in G\,|\,x^{p^t}=1\}
\textrm{\ and\ }
\mho_t(G)=\{x^{p^t}\,|\,x\in G\}.\]
Furthermore, in this case $G/\Omega_t(G)\simeq \mho_t(G)$, 
and $\mho_t(G)$ equals the $t$-th term of the Frattini series of $G$.  

% -----------------
We improve another related result of Mann \cite[Theorem 3]{mann2} as follows
%-------------------
\begin{thm}\label{thm:Mann-G:G2}
  Let $G$ be a $2$-group and let $E$ be a maximal normal elementary
  abelian subgroup of $G$. If $d(E)=k$, then
  $|G:\mho_1(G)|\leq 2^{\frac{k(k+5)}2}$.
\end{thm}
%-------------------------------------------
%Another natural question is that what is the connection between
%$r(G)$ and $nr(G)$? Of course, $nr(G)\leq r(G)$ always holds.
%Alperin and Glauberman \cite[Theorems A and
%5.8]{alperin-glauberman} gave several conditions, which imply that
%$nr(G)=r(G)$.

Knowing $d(E)$ for a maximal normal elementary abelian subgroup
$E\nor G$ not only gives restrictions on $d(H)$ for subgroups $H$
of $G$, but on the structure of subgroups of $G$ more deeply. A
particularly interesting question could be that what can be said about
the cyclic decomposition of an abelian subgroup $A$ of $G$ if such an
information is known.  In this paper we prove that under the same
assumption as of Theorem \ref{thm:Th-Mann}, the number of factors in
the cyclic decomposition of $A$ which are larger than $C_p$ is more
restricted.
%-----------------
\begin{thm}\label{thm:d_PhiA_leq_2d_PhiE}
  Let $G$ be a finite $p$-group and let $E$ be a maximal 
  normal elementary abelian subgroup of $G$.
  If $d(E)=k$, and $A$ is any abelian subgroup of $G$, then 
  \begin{enumerate}
  \item $d(\Phi(A))\leq 2k$ for $p>2$,
  \item $d(\Phi(A))\leq 3k$ for $p=2$.
  \end{enumerate}
\end{thm}
% ---------------------------------

It seems quite possible that if the stronger condition $nr(G)\leq k$
holds, then the number of generators of any abelian subgroup $A$ is at
most linear in $k$ (see Question \ref{que:d(A)_leq_2max_d(E)}). 
By Theorem \ref{thm:d_PhiA_leq_2d_PhiE}, such a bound holds for the 
generating number of the Frattini subgroup of any abelian subgroup of $G$.
As another piece of evidence let us quote the following:
%------------------------
\begin{thm}[Alperin, Glauberman \cite{alperin-glauberman}]
  Let $G$ be a finite $p$-group satisfying one of the following conditions.
  \begin{enumerate}
  \item $p$ is odd and $p>4 r(G)-7$;
  \item $G$ has nilpotency class at most $p$;
%  \item $G$ has exponent $p$ and nilpotency class at most $p+1$;
%  \item $p\leq 3$ and $G$ has exponent $p$.
  \end{enumerate}
  Then $nr(G)=r(G)$.
\end{thm}
%------------------------
In contrast, examples of Alperin \cite[Exercise 31, p.~349]{huppert}
and Glauberman \cite{glauberman} shows that $nr(G)$ can be strictly
smaller than $r(G)$.

%On the other hand, $r(G)$ is bounded by a quadratic function of
%$nr(G)$ by Theorem \ref{thm:Th-Mann}.  By our knowledge, results of
%these type always give upper bound for $r(G)$ in terms of $d(E)$ for
%any maximal normal abelian subgroup $E$ of $G$, but by Example
%\ref{ex:Th-Mann_best}, $G$ can have maximal normal elementary abelian subgroups
%of very different rank, so it might happen that $r(G)$ is bounded even
%by a linear function of $nr(G)$ (see Question
%\ref{que:d(A)_leq_2max_d(E)}).
% 
% ---------------------------------
A key result in this paper (which is essential for the proof of
Theorem \ref{thm:d_PhiA_leq_2d_PhiE}) says that a commutative
subalgebra $\A\leq \Hom(V)$ is ``close to being a zero algebra'' in the
following sense.
%-----------------
\begin{thm}\label{thm:big_zerosubal}
  Let $V$ be an $n$-dimensional vector space over the field $K$ and 
  let $\A\leq\Hom(V)$ be a commutative algebra. Then there exists 
  a zero algebra $\B\leq \A$ satisfying $\codim(\B,\A)\leq n$.
\end{thm}
%-----------------

As a consequence of Theorem \ref{thm:big_zerosubal}
we have the following 
%-----------------
\begin{thm}\label{thm:d_Phi_G}
  Let $A\leq GL(n,p)$ be an abelian subgroup.  Then
  $|A:\Omega_1(O_p(A))|\leq p^n$. In particular, if $A\leq GL(n,p)$ is an
  abelian $p$-subgroup then there are at most $n$ many factors in the
  cyclic decomposition of $A$, which are larger than $C_p$.
  In other words, $d(\Phi(A))\leq n$ for any abelian $p$-subgroup
  $A\leq GL(n,p)$.
\end{thm}
% -----------------
\begin{rem}
  Let $k=n/(p+1)$ and $V=V_1\oplus\ldots \oplus V_k$ with
  $\dim(V_i)=p+1$ for each $i$. Furthermore, let
  $g_1,\ldots,g_k\in GL(V)$ be such that $g_i|_{V_j}=\id_{V_j}$ for
  $i\neq j$ while $g_i|_{V_i}$ corresponds to a unipotent Jordan-block
  for each $i$. Then $o(g_i)=p^2$ and
  $\langle g_1,\ldots, g_k\rangle =\langle g_1\rangle \times \langle
  g_2\rangle\times\ldots\times \langle g_k\rangle\simeq C_{p^2}^k$, so
  the upper bound in Theorem \ref{thm:d_Phi_G} is essentially the
  best possible.
\end{rem}
% -----------------

An immediate consequence is that $|A|\leq p^n-1$ holds for any abelian
$p'$-subgroup of $GL(n,p)$. Previously, we only knew of a proof for
this fact which depends on Maschke's theorem.  (Note that a subgroup
of $GL(n,p)$ generated by a Singer cycle has order exactly $p^n-1$, so
this bound is the best possible.)

%=====================================================================
\section{Proofs}
Let $V$ be an $n$ dimensional vector space over a field $K$ and 
let $\A\leq \Hom(V)$ be a commutative subalgebra in the full 
endomorphism algebra $\Hom(V)$ of $V$. We use the notation 
$\A^2$ for the subalgebra of $\A$ generated by all products 
$\{xy\,|\,x,y\in \A\}$. Furthermore, let 
$\ker(\A):=\cap_{a\in \A}\ker(a)=\{v\in V\,|\,a(v)=0\ \forall a\in \A\}$. 
Clearly, $\A$ is a zero algebra if and only if $\ker(\A^2)=V$. 
In what follows, for two subspaces $U\leq W$, the codimension of 
$U$ in $W$ is denoted by $\codim(U,W):=\dim(W)-\dim(U)$.
%----------------------
First, we prove the following stronger theorem than
Theorem \ref{thm:big_zerosubal}.
Note that in our terminology the property ``ideal'' also 
implies ``$K$-subspace''.
%----------------------
\begin{thm}
  Let $V$ be an $n$ dimensional vector space over the field $K$ and
  let $\A\leq\Hom(V)$ be a commutative algebra with $\dim(\ker(\A))=k$. 
  Then there is an ideal $\B$ of $\A$ satisfying 
  $\codim(\B,\A)\leq n-k$ and $\B^2=0$.
\end{thm} 
%-------------
\begin{proof}
  We define a series of integers 
  $0=l_0<l_1<\ldots\leq n-k$ and ideals 
  $\A=\A_0>\A_1>\ldots$ of $\A$ such that $\codim(\A_i,\A)=l_i$ and 
  $\dim(\ker(\A_i))\geq l_i+k$ holds for every $i$. 
  For $i=0$, the pair $l_0=0,\ \A_0=\A$ clearly satisfies both conditions. 
  Let us assume that we found the pair $l_i,\A_i$ for some $i$. Now, 
  if $\ker(\A_i^2)=V$, then statement of the Theorem holds for $\B:=\A_i$.

  Otherwise, let us choose an $x\in V$ such that
  $x\notin \ker(\A_i^2)$. This means that
  $V_i=\A_i(x):=\{a(x)\,|\,a\in \A_i\}$ is not contained in
  $U_i:=\ker(\A_i)$. Since $\A_i$ is an ideal of $\A$, both $U_i$ and
  $V_i$ are $\A$-invariant, so $U_i\cap V_i$ is also
  $\A$-invariant. Now, let
  \[
  \A_{i+1}=\{a\in \A_i\,|\,a(x)\in U_i\cap V_i\},\ 
  m_i:=\codim(U_i\cap V_i,V_i)>0,\ l_{i+1}:=l_i+m_i.
  \]
  Since $\varphi_x:a\mapsto a(x)$ defines a surjective linear map 
  $\varphi_x:\A_i\mapsto V_i$ and $\A_{i+1}=\varphi_x^{-1}(U_i\cap V_i)$ 
  it readily follows that $\codim(\A_{i+1},\A_i)=m_i>0$, so
  $\codim(\A_{i+1},\A)=l_{i+1}>l_i$. Furthermore, the $\A$-invariance of 
  $U_i\cap V_i$ and $\A_i\nor \A$ implies that $\A_{i+1}\nor \A$.  

  It remains to prove that $\dim(\ker(\A_{i+1}))\geq l_{i+1}+k$.
  (Thus, $l_{i+1}\leq n-k$ also holds!) For any $y\in V_i$ we have
  $y\in \A_i(x)$, so
  $\A_{i+1}(y)\subset \A_{i+1}(\A_i(x))=\A_i(\A_{i+1}(x))\subset
  \A_i(U_i)=0$ by using the commutativity of $\A$ and the definition
  of $\A_{i+1}$ and $U_i$.  Therefore, $V_i\leq \ker(\A_{i+1})$. On
  the other hand, $U_i=\ker(\A_i)\leq \ker(\A_{i+1})$, so
  $\dim(\ker(\A_{i+1}))\geq \dim(U_i+V_i)= \dim(U_i)+\codim(U_i\cap
  V_i,V_i)\geq l_i+k+m_i=l_{i+1}+k$ also holds.

  Trivially, the series $0=l_0<l_1<\ldots\leq n-k$ has length at most 
  $n-k+1$, so we find a sufficient ideal $\B=\A_i$ for some $i$ 
  in at most  $n-k+1$ many steps.
\end{proof}
%---------------------------
\begin{proof}[Proof of Theorem \ref{thm:d_Phi_G}]
  We only need to prove the first statement, since 
  the second statement is just a special case of the first.

  Let $V$ be an $n$ dimensional vector space over $\FF{p}$, so we can
  view $A$ as a subgroup of $GL(V)$. If $A\leq B\leq GL(V)$ and $B$ is
  also an abelian subgroup, then $\Omega_1(O_p(A))=A\cap \Omega_1(O_p(B))$, so
  $|A:\Omega_1(O_p(A))|\leq |B:\Omega_1(O_p(B))|$. Thus, we can assume that
  $A\leq GL(V)$ is maximal among the abelian subgroups of $GL(V)$. Let
  $\A\leq \Hom(V)$ be the subalgebra of $\Hom(V)$ generated by $A$.
  Then $\A$ is commutative and $A=U(\A)$ is the unit group of $\A$.  By
  Theorem \ref{thm:big_zerosubal}, there is a zero algebra $\B\leq \A$
  with $\codim(\B,\A)\leq n$. Then $1+\B\leq A$ is an elementary abelian
  $p$-subgroup of $A$, so $1+\B\leq \Omega_1(O_p(A))$. Therefore,
  \[
  |A:\Omega_1(O_p(A))|\leq \frac{|A|}{|1+\B|}\leq \frac{|\A|}{|\B|}=
  p^{\codim(\B,\A)}\leq p^n,
  \] 
  and the claim follows.
\end{proof}
%------------------------------
Before the proof of Theorem \ref{thm:d_PhiA_leq_2d_PhiE}, we 
summarise an idea of Sambale, which can be found in the proof of 
\cite[Theorem 1.3]{sambale}). This idea will also be used in the proofs of 
 Theorem \ref{thm:Th-Mann_p=2} and Theorem \ref{thm:Mann-G:G2}.

 Let $G$ be a finite $2$-group and let $E$ be a maximal elementary
 abelian normal subgroup of $G$ with $d(E)=k$. Let $C=C_G(E)$.  Choose
 a maximal abelian normal subgroup $A$ of exponent at most $4$ which
 contains $E$. Then obviously $C_G(A)\leq C$. By a result of Alperin
 \cite[Theorem]{alperin} (see also \cite[III, 12.1 Satz]{huppert}),
 $\Omega_2(C_G(A))=A\leq Z(C_G(A))$, that is, $C_G(A)$ is
 $2$-central. (For the definition and basic properties of $p$-central
 groups see \cite{bubbolini-corsi} and \cite{mann2}.)  Sambale
 observed that $C/C_G(A)$ is elementary abelian. Furthermore, by using
 a theorem of MacWilliams (see \cite[Theorem 37.1]{berkovich}), he
 showed that $|C:\Phi(C)|\leq 2^{2k}$.  We note that Sambale's
 argument can be used without modification to prove that
 $|H:\Phi(H)|\leq 2^{2k}$ holds for any subgroup $H$ satisfying
 $E\leq H\leq C$.
%------------------------------
\begin{proof}[Proof of Theorem \ref{thm:d_PhiA_leq_2d_PhiE}]
  First, we consider the case $p>2$. In accordance with the
  assumption, let $E$ be a maximal elementary abelian normal subgroup
  of $G$ with $d(E)=k$. Let $C$ be the centraliser of $E$ in $G$. Then
  $\Omega_1(C)\leq E$ by \cite[Theorem]{alperin}. If $A\leq G$ is
  any abelian subgroup, then $d(A\cap C)=d(\Omega_1(A\cap C))\leq
  d(E)=k$ holds.  The action of $A$ on $E$ defines
  an injection $A/A\cap C\mapsto \aut(E)\simeq GL(k,p)$, so
  $d(\Phi(A/A\cap C))\leq k$ by Theorem \ref{thm:d_Phi_G}. 
  Therefore, 
  \[
  d(\Phi(A))\leq d(\Phi(A)(A\cap C))\leq d(\Phi(A/A\cap C))+d(A\cap C)\leq 2k.
  \]

  Now, we turn to the case $p=2$. Let $E$ be a maximal elementary
  abelian normal subgroup of $G$ with $d(E)=k$ and $C=C_G(E)$.
  Furthermore, let $A\leq G$ be any abelian subgroup.  Using the
  aforementioned result of Sambale we get $d(A\cap C)\leq d((A\cap
  C)E)\leq 2k$.  On the other hand, the same argument as in case $p>2$
  proves that $d(\Phi(A/A\cap C))\leq k$. Therefore,
  \[
  d(\Phi(A))\leq d(\Phi(A/A\cap C))+d(A\cap C)\leq k+2k=3k.
  \]
\end{proof}
% ----------------------------------------------
Now, we show that MacWilliams' Theorem \cite[Theorem A]{macwilliams}
can be extended for $p=2$, as well.
% ----------------------------------------------
\begin{proof}[Proof of Theorem \ref{thm:Macwilliams_p=2}]
  We only point out, how MacWilliams' argument must be modified to
  hold also for $p=2$. MacWilliams proof can be divided into two parts.
  \begin{enumerate}
    \item First, she proves that if $p$ is an odd prime and $G$ is a $p$-group, 
      then there is a subgroup $H\leq G$ of nilpotency class at most two 
      satisfying $d(G)\leq d(H)$.
    \item Second, starting from a $p$-subgroup $G\leq GL(n,p)$ of
      nilpotency class at most two, she modify it to get a $\tilde
      G\leq GL(n,p)$ with $d(G)\leq d(\tilde G)$ such that $d(\tilde
      G)$ can easily be calculated.
  \end{enumerate}
  It turns out that part (2) of MacWilliams' proof works also for the
  case $p=2$, but the claim in part (1) is not valid for $p=2$.
  However, there is a similar statement which also follows for $p=2$.
  Let $w=x^{p^2}[y,z]\in F_3$ be a word, i.e. an element of the free
  group $F_3=\langle x,y,z\rangle$. For any $p$-group $P$, let
  $w(P)=\langle w(g_1,g_2,g_3)\,|\, g_1,g_2,g_3\in P\rangle$ be the
  verbal subgroup of $P$ defined by $w$. Thus, $s(P):=|P:w(P)|$ equals
  the order of the largest abelian quotient of $P$ with exponent at
  most $p^2$.  Now, results of Gonz\'alez-S\'anchez and Klopsch
  (\cite[Lemma 3.1]{gonzalez-klopsch} and \cite[Theorem
  3.3]{gonzalez-klopsch}) imply that there is a subgroup
  $G_1\leq G\leq GL(n,p)$ of nilpotency class $\leq 2$ such that
  $s(G_1)= s(G)$. Now, the initial step of MacWilliams' modifying
  argument can be used to find a $G_2=N\rtimes H\leq GL(n,p)$ of
  nilpotency class $\leq 2$ such that $|G_2|=|G_1|$, furthermore
  $N\nor G_1$ and $\Phi(G_2)=G_2'=[N,G_2]=[N,G_1]\leq G_1'$.  (For
  details, see \cite[page 135]{macwilliams}.)  Therefore,
  \[s(G)=s(G_1)\leq |G_1:G_1'|\leq |G_2:G_2'|=|G_2:\Phi(G_2)|=p^{d(G_2)}.\]
  Now, $d(G_2)\leq \frac{1}{4}n^2$ by MacWilliams' argument, so 
  $s(G)\leq p^{\frac{1}4n^2}$ which readily implies $d(G)\leq \frac{1}{4}n^2$. 
\end{proof}
%-----------------------------------
\begin{proof}[Proof of Theorem \ref{thm:Th-Mann_p=2}]
  Let $C=C_G(E)$ be the centraliser of $E$ and let $H\leq G$ be any
  subgroup of $G$. Using Sambale's result to the
  group $(H\cap C)E$ we get that $d(H\cap C)\leq d((H\cap C)E)\leq 2k$.  On
  the other hand, $H/(H\cap C)$ is included in
  $\aut(E)\simeq GL(k,2)$, so $d(H/(H\cap C))\leq \frac{1}4k^2$ by
  Theorem \ref{thm:Macwilliams_p=2}.  Therefore,
  $d(H)\leq d(H\cap C)+d(H/(H\cap C))\leq 2k+\frac{1}4k^2$, as
  claimed.
\end{proof}
%----------------------------------------------
\begin{proof}[Proof of Theorem \ref{thm:Mann-G:G2}]
  Let $E$ be a maximal elementary abelian normal subgroup of $G$ with
  $d(E)=k$ and let $C=C_G(E)$. Then $G/C$ is a subgroup of
  $\aut(E)\simeq GL(k,2)$, so $|G:C|\leq 2^{\binom{k}{2}}$. As in
  Sambale's argument (see the paragraph preceding the proof of Theorem
  \ref{thm:d_PhiA_leq_2d_PhiE}) choose a maximal abelian normal
  subgroup $A$ of exponent at most $4$ which contains $E$.  Then
  $C/C_G(A)$ is elementary abelian, and $|C:\Phi(C)|\leq 2^{2k}$, so
  $|C:C_G(A)|\leq |C:\Phi(C)|\leq 2^{2k}$. Furthermore, $C_G(A)$ is
  $p$-central, so, by using \cite[Proposition 4]{mann2}, we get that
  $|C_G(A):\mho_1(C_G(A))|\leq |\Omega_1(C_G(A))|=|E|=2^k$.
  Therefore, 
  \begin{align*}
  |G:\mho_1(G)|&\leq |G:\mho_1(C_G(A))|=
  |G:C|\cdot |C:C_G(A)|\cdot |C_G(A):\mho_1(C_G(A))|\\
  &\leq 2^{\binom{k}2+2k+k}=2^{\frac{k(k+5)}2}.
  \end{align*}
\end{proof}
%==============================================
\section{Related problems}
%---------------------
In this section we pose some problems related to the above results. 
A positive answer to the following question would be a generalization
of Theorem \ref{thm:d_Phi_G}.
% -----------------
\begin{que}\label{que:section}
  Let $G\leq GL(n,p)$ be a $p$-group and let $H/K$ be an abelian
  section of $G$, that is, $K\nor H\leq G$ with $H/K$ abelian. Is is
  true that there are at most $n$ many factors in the cyclic
  decomposition of $H/K$, which are larger than $C_p$? Or, at least, 
  is the number of such factors bounded by $O(n)$?
\end{que}
% --------------------
Another possible generalisation of Theorem \ref{thm:d_Phi_G} is
\begin{que}
  Let $G\leq GL(n,p)$ be $p$-central. Is it true that
  $|\Omega_2(G)/\Omega_1(G)|\leq p^n$?
\end{que}
% ----------------
By \cite[Lemma C]{mann2}, the $p$-central assumption implies that
$\Omega_{2}(G)$ is of exponent $p^2$ and of nilpotency class $2$.
Furthermore,
$|\Omega_{i+1}(G)/\Omega_i(G)|\leq |\Omega_2(G)/\Omega_1(G)|$ for
every $i\geq 2$ in a $p$-central group.

% ---------------------
Note that if $G\leq GL(n,p)$ is a $p$-Sylow subgroup of $GL(n,p)$,
then $d(\Phi(G))=2n-5$, so the final conclusion in Theorem
\ref{thm:d_Phi_G} does not remain valid if the abelian condition for
$G$ is dropped. In fact, the below example shows that there exists a
$p$-group $G\leq GL(n,p)$ such that $d(\Phi(G))$ is roughly $n^2/4$.
%---------------
\begin{example}
  Let $\{E_{ij}\,|\,1\leq i,j\leq n\}$ be the usual basis of the vector space of
  $n\times n$ matrices over $\FF p$ and let $1$ denote the $n\times n$
  identity matrix. 
  Let 
  \[
    G=\Big\{1+\sum_{i<j} a_{ij}E_{ij}\,\Big|\,a_{ij}=0\textrm{ if }
    j<\lceil n/2\rceil
    \textrm{ or }i>\lceil n/2\rceil\Big\}\leq GL(n,p).
  \]
  Then
  \[
    \Phi(G)=G'=Z(G)=
    \Big\{1+\sum_{i<j} a_{ij}E_{ij}\,\Big|\,a_{ij}=0\textrm{ if }
    j\leq\lceil n/2\rceil
    \textrm{ or }i\geq \lceil n/2\rceil\Big\}
  \]
  has rank $\lfloor n/2\rfloor\cdot (\lceil n/2\rceil-1)$.
\end{example}
% ---------------
Maybe the abelian condition in Theorem \ref{thm:d_Phi_G} can be
weakened to several important classes of $p$-groups. We ask
% ---------------
\begin{que}
  Let $G\leq GL(n,p)$ be a $p$-central, powerful or regular $p$-group.
  Is it true that
  $d(\Phi(G))\leq n$ (or $O(n)$)?
\end{que}
The next problem is similar to Theorem \ref{thm:big_zerosubal}. It
might be useful to answer Question \ref{que:section}.
%-----------------
\begin{que}\label{que:algebra}
  Let $V$ be an $n$ dimensional vector space over the field $K$ and
  let $A \leq \Hom(V )$ be a nilpotent algebra. Does there exist a
  subalgebra $B \leq A$ of codimension at most $n$ (or $O(n)$) such
  that $B^2\leq [B,B]$ (or, at least, $B^2\leq [A,A]$)?
\end{que}
% ------------------------
% Olshanskii konstrukciójának "felfújása".
One might think that Question \ref{que:section} could be 
%proved in the affirmative 
reduced to Theorem \ref{thm:d_Phi_G} by showing
that if a finite $p$-group has a quotient
isomorphic to $(C_{p^r})^{n}$ for some $r$ and $n$, then it always
contains a subgroup isomorphic to $(C_{p^r})^{n}$ (or, at least
$(C_{p^r})^{\varepsilon n}$ for some absolute constant
$\varepsilon>0$). However, this is not the case; For $r=1$, this has
been proved by Ol'shanskii \cite{olshanskii}. Using his result, we
now prove Theorem \ref{thm:Endre_konstrukcio}, which is a 
generalisation of the above statement for any $r\geq 1$.

% ----------------------
First we prove a lemma.
% ----------------------
\begin{lem}\label{lem:bilinear->group}
  Let $R$ be a commutative ring, $A=R^n$, $B=R^k$ and let
  $\varphi:A\times A\mapsto B$ be an alternating $R$-bilinear map.
%i.e. for all $x,y,z\in A,\,r\in R$ we have
%\begin{gather*}
%\varphi(x,y)=-\varphi(y,x),\quad \varphi(x+y,z)=\varphi(x,z)+\varphi(y,z),\\
%\varphi(r\cdot x,y)=r\cdot\varphi(x,y),\quad
%\varphi(x\cdot r,y)=\varphi(x,y)\cdot r.
%\end{gather*}
  Then there is a $2$-nilpotent group $G$ and $G'\leq N\leq Z(G)$ such
  that $G/N\simeq A,\ N\simeq B$ as abelian groups and the commutator
  map $[.\,,\,.]: G/N\times G/N\mapsto N,\ (xN,yN)\mapsto [x,y]$
  agrees with $\varphi$ under these isomorphisms.
\end{lem}
%--------------
\begin{proof}
  First, if $S$ is any ring with $S^3=0$, then $G:=1+S$ is a $2$-nilpotent
  group with group operation $(1+s)(1+t):=1+s+t+st$ satisfying 
  $G'\leq 1+S^2\leq Z(G)$. Furthermore, 
  $[1+s,1+t]=(1-s+s^2)(1-t+t^2)(1+s)(1+t)=1+st-ts$ holds for every $s,t\in S$. 
  
  Now, starting from $A,B,\varphi$ we construct a ring $S$ with
  underlying abelian group $A\oplus B$.  Let $e_1,\ldots, e_n$ be the
  canonical basis of $A$. We define the multiplication on $S$ as
  \[
  BS=SB=0,\ e_ie_j=\left\{
    \begin{array}{ll}
      \varphi(e_i,e_j)&\textrm{if }i<j,\\
      0              &\textrm{if }i\geq j
    \end{array}
    \right.
  \]
  and we extend it to the whole $S$ in  a distributive way. 
  Then $S$ is a ring with $S^3=0$, so $G=1+S$ is a group and $N=1+B$ satisfies
  $G'\leq 1+S^2\leq N\leq Z(G)$. Furthermore, the maps 
  $(1+a)N\mapsto a$ and $1+b\mapsto b$ ($a\in A,\ b\in B)$ define isomorphisms 
  $G/N\mapsto A$ and $N\mapsto B$, respectively. 
  Finally, for every $1\leq i,j\leq n$ we have 
  \[
  [(1+e_i)N,(1+e_j)N]=1+e_ie_j-e_je_i=1+\varphi(e_i,e_j). 
  \]
  Thus the commutator map 
  $[.\,,\,.]: G/N\times G/N\mapsto N$ agrees with  $\varphi$ 
  on the set of generators $\{(1+e_i)N\,|\,1\leq i\leq n\}$ 
  under the above isomorphisms, so it agrees with $\varphi$ on the whole $G/N$. 
\end{proof}
% ----------------------------------------------
\begin{rem}\leavevmode
\begin{enumerate}
\item The above construction also works in the more general case if
  $A$ is any (not necessarily finite dimensional) free $R$-module and 
  $B$ is any $R$-module. 
\item If there is a half of every element in $B$ (for example, when
  $R$ is a $K$-algebra over a field $K$ of characteristic different
  from $2$), then the multiplication $A\times A\mapsto B$ can be
  defined in a more natural way by choosing
  $a_1a_2:=\frac{1}{2}\varphi(a_1,a_2)$. In that case the exponent of
  $G$ always agrees with the exponent of $R$ as an additive group.
\end{enumerate}
\end{rem}
% ----------------------------------------------
\begin{proof}[Proof of Theorem \ref{thm:Endre_konstrukcio}]
  Let $\tilde{A}:=\ZZ_p^n,\ \tilde{B}:=\ZZ_p^k$ and let 
  $\tilde{\varphi}:\tilde A\times \tilde A\mapsto \tilde B$ be an
  alternating bilinear map such that there is no $k$-dimensional
  completely isotropic subspace of $A$ with respect to
  $\tilde{\varphi}$. (Since $2n<k(k-1)$, such a map exists by
  \cite[Lemma 2.]{olshanskii}.) Let $M_{\tilde\varphi}\in
  (\tilde{B})^{n\times n}$ be the matrix form of $\tilde\varphi$ with
  respect to the natural basis of $\tilde A=\ZZ_p^n$ so $M_{\tilde\varphi}$ is 
  an alternating matrix over $\tilde{B}$.

  Let us choose $R=\ZZ_{p^r},\ A=R^n,\ B=R^k$, so $A\simeq C_{p^r}^n$
  and $B\simeq C_{p^r}^k$ as abelian groups.  Let $M_\varphi\in
  B^{n\times n}$ be an alternating matrix over $B$ such that the
  natural homomorphism $\pmod p:\ZZ_{p^r}\mapsto \ZZ_{p}$ maps
  $M_\varphi$ to $M_{\tilde\varphi}$ and let $\varphi: A\times A\mapsto B$ 
  be the alternating map whose matrix is $M_\varphi$ with respect to the 
  natural basis of $A$. 

  By Lemma \ref{lem:bilinear->group} and its proof, there are
  $p$-groups $G=G(A,B,\varphi)$ and
  $\tilde{G}=G(\tilde A,\tilde{B},\tilde\varphi)$ of the form $G=1+S$
  and $\tilde{G}=1+\tilde{S}$.  By construction, the $\pmod p$-map
  extends to a surjective ring homomorphism $S\mapsto \tilde{S}$, so
  it also defines a surjective group homomorphism $\rho:G\mapsto
  \tilde{G}$ whose kernel is $K=1+pS$. Now, for any $s\in S$ we have
  $(1+ps)^{p^{r-1}}=1+p^rs+\binom{p^{r-1}}{2}\cdot p^2 s^2=1$, so the
  exponent of $K$ is $p^{r-1}$.

  It remains to prove that $G$ does not contain any abelian subgroup
  isomorphic to $C_{p^r}^{2k}$. Assuming the converse, let $H\leq G$
  be such a subgroup.  Then $\rho(H)\simeq H/H\cap K$ is an abelian
  subgroup of $\tilde{G}$ such that $d(\rho(H))=2k$. Therefore, the
  image of $\rho(H)$ under the natural map $\tilde G\mapsto \tilde
  G/\tilde B\simeq \tilde A$ is a completely isotropic subspace with
  respect to the form $\varphi$ whose dimension is at least $k$, which
  is a contradiction.
\end{proof}
% ----------------------------------------------
Some results from \cite{gonzalez-klopsch} suggest that
Question \ref{que:section} might be reduced to $p$-groups of nilpotency
class $2$ as follows.
% ----------------------------------------------
\begin{que}
  Let $G$ be a finite $p$-group such that $G$ has a quotient isomorphic to 
  $(C_{p^r})^l$ for some positive integers $l$ and $r>1$. Is it true that
  $G$ contains a subgroup of nilpotency class at most $2$ with this property?
\end{que}
% ----------------------------------------------
By Example \ref{ex:Th-Mann_best}, if we only assume that $G$ contains
a maximal abelian normal subgroup $A$ with $d(A)=k$, then
$O(k^2)$ is the smallest general upper bound to $r(G)$. On the other
hand, if we assume that $d(A)\leq k$ for \emph{every} maximal abelian normal
subgroup $A$ of $G$ (i.e. we assume that $nr(G)\leq k$), then we do
not know any similar example. So we may ask:
% ----------------------------------------------
\begin{que}\label{que:d(A)_leq_2max_d(E)}
  Let $G$ be a $p$-group and let us assume that $nr(G)\leq k$, that
  is, $d(A)\leq k$ for every abelian normal subgroup $A$ of $G$.  Is
  it true that $r(G)\leq 2k$, that is, $d(B)\leq 2k$ holds for every
  abelian subgroup $B$ of $G$ ?
\end{que}
%----------------------------------------------
\begin{rem}\label{rem:d(A)_leq_2max_d(E)}
  The $k$-term direct power $D_{16}^k$ (where $D_{16}$ is the dihedral
  group of order $16$) shows that this bound is the best possible. 
\end{rem}
%----------------------------------------------
One can ask a similar question, but using the order of abelian
subgroups instead of their rank.
%----------------------------------------------
\begin{que}\label{que:|A|_leq_max_|E|^2}
  Let $G$ be a $p$-group and let us assume that $|A|\leq p^m$ for
  every abelian normal subgroup $A$ of $G$. Is is true that 
  $|B|\leq p^{2m}$ holds for every abelian subgroup $B$ of $G$? 
\end{que}
%----------------------------------------------
\begin{rem}
  Examples of Alperin and Glauberman \cite[Exercise 31, p.~349]{huppert},
  \cite{glauberman4} show that there exists a
  $p$-group $G$ for which
  $\max\{|B|\,|\,B\leq G \textrm{ is abelian}\}$ is strictly larger
  than $\max\{|A|\,|\,A\nor G \textrm{ is abelian}\}$.  Moreover, if
  $p\geq 5$, then there exists a group of exponent $p$ with this
  property.

  On the other hand, under various conditions (for example when $G$ is
  metabelian \cite{gillam} or it has nilpotency class at most $p-1$
  \cite{glauberman2}) there is a normal abelian subgroup among the
  abelian subgroups of maximal order.
\end{rem}
%----------------------------------------------

\end{document}